\documentclass{amsart}
\usepackage[utf8]{inputenc}
\usepackage[english]{babel} 
\usepackage{amsmath} 
\usepackage{dsfont}
\usepackage{centernot}
\usepackage{amssymb}
\usepackage{mathtools}
\usepackage{amsthm}
\usepackage{mathrsfs}
\usepackage{enumerate}
\usepackage{xcolor}
\usepackage{hyperref}
\usepackage{bbm}
\usepackage{comment}
\usepackage{indentfirst}
\usepackage{scalerel}[2016/12/29]
\usepackage{wasysym}
\usepackage{stmaryrd}
\usepackage{makecell}
\usepackage{tikz-cd}

\DeclareMathOperator{\cl}{cl}
\DeclareMathOperator{\ran}{ran}

\DeclareMathOperator{\supp}{supp}
\DeclareMathOperator{\ext}{ext}
\DeclareMathOperator{\isom}{Isom}
\DeclareMathOperator{\spn}{span}
\DeclareMathOperator{\sgn}{sgn}
\DeclareMathOperator{\conv}{conv}

\renewcommand{\epsilon}{\scaleobj{1.3}{\varepsilon}}

\newcommand {\R}{\mathbb{R}}
\newcommand {\N}{\mathbb{N}}

\newcommand {\K}{\mathbb K}

\makeatletter
\def\moverlay{\mathpalette\mov@rlay}
\def\mov@rlay#1#2{\leavevmode\vtop{%
   \baselineskip\z@skip \lineskiplimit-\maxdimen
   \ialign{\hfil$\m@th#1##$\hfil\cr#2\crcr}}}
\newcommand{\charfusion}[3][\mathord]{
    #1{\ifx#1\mathop\vphantom{#2}\fi
        \mathpalette\mov@rlay{#2\cr#3}
      }
    \ifx#1\mathop\expandafter\displaylimits\fi}
\makeatother

\newcommand {\per}{\mathbb P}
\newcommand {\permax}{\mathbb P^{\rm{MAX}}}

\DeclareRobustCommand{\rchi}{{\mathpalette\irchi\relax}}
\newcommand{\irchi}[2]{\raisebox{\depth}{$#1\chi$}} 

\numberwithin{equation}{section}

\theoremstyle{plain}
\newtheorem{theo}{Theorem}
\newtheorem{prop}[theo]{Proposition}
\newtheorem{lemma}[theo]{Lemma}

\theoremstyle{definition}
\newtheorem{defin}[theo]{Definition}
\newtheorem{example}[theo]{Example}

\theoremstyle{remark}
\newtheorem{rem}[theo]{Remark}

\numberwithin{theo}{section}

\title{Isometries of James-Schreier and Lorentz spaces}
\author{Christina Brech}
\address{Departamento de Matemática, Instituto de Matemática e Estatística, Universidade de São Paulo,
Rua do Matão, 1010, 05508-090, São Paulo, Brazil}
\email{brech@ime.usp.br}

\author{Victor dos Santos Ronchim}
\thanks{The first author was partially supported by FAPESP grants (2016/25574-8 and 2023/12916-1). The second author was supported by CNPq (150193/2022-0).} 
\address{Departamento Bioprocessos e Biotecnologia, Faculdade de Ciêncas Agronômicas, Universidade Estadual Paulista,
Av. Universitária, 3780, 18610-034, Botucatu, Brazil}
\email{victor.ronchim@unesp.br}

\date{\today}
\begin{document}
\begin{abstract}
We study the group of surjective linear isometries on certain real Banach sequence spaces using the preservation of extreme points in the closed unit ball. Our main result provides a characterization of the extreme points of the dual unit ball of the James-Schreier space $V_1$. As a consequence, we show that the only isometries on $V_1$ are $\pm Id$. We also obtain a Banach-Stone-type result for Lorentz sequence spaces, analogous to one proved in \cite{CarothersIsometrisLorentz} for Lorentz function spaces.
\end{abstract}
    
    

\maketitle
\section{Introduction}
One of the main topics in the development of mathematics is the study of isometries. They have been studied since the early days of geometry, with a focus on operators preserving lengths and angles. In the class of infinite dimensional Banach spaces, the characterization of isometries dates back to the publication of Banach's book \cite{banach1932theorie} in 1932, where all surjective isometries on some classical funtion spaces, such as $C[0,1]$, $L_p$ for $p\geq 1$, were described, and also stated that similar results would hold for the sequence spaces $c_0$ and $\ell_p$, without presenting the arguments. In this paper, we focus our attention on two types of real Banach sequence spaces: James-Schreier spaces and Lorentz spaces. Our main goal is to characterize the isometries of these spaces in terms of permutations of their canonical basis and changes of signs. 

Given a real Banach space $X$, let us denote by $\isom(X)$ the group of all surjective linear isometries $T:X\rightarrow X$. We say that $X$ has a \emph{trivial} group of isometries if every isometry is a multiple of the identity map $Id$, i.e., $\isom(X)=\{\pm Id\}$. If $X$ is a sequence space, following \cite{AntunesBeanlandsurvey}, we say that the group of isometries of $X$ is \emph{diagonal} if it consists of all operators $T: X \rightarrow X$ of the form:
\begin{equation}\label{definition diagonal group isom}
T(x_n) = (\epsilon_n x_{n})
\end{equation}
for some sequence of signs $(\epsilon_n)\in \{-1,1\}^{\N}$. In this case, $\isom(X)\cong \mathbb Z_2^\N$. Finally, $\isom(X)$ is \emph{standard} if it consists of all operators $T: X \rightarrow X$ of the form:
\begin{equation}\label{definition standard group isom}
T(x_n) = (\epsilon_n x_{\pi(n)})
\end{equation}
for some $(\epsilon_n)\in \{-1,1\}^{\N}$ and some $\pi\in S_{\infty}$, where $S_{\infty}$ denotes the set of all permutations of the natural numbers. Observe that if $X$ has a standard group of isometries, then $\isom(X)$ is isomorphic to a semidirect product of  $ Z_2^\N$ and $S_\infty$. Analogous definitions can be considered in the complex case, by replacing $\{-1,1\}$ by the unit circle in $\mathbb{C}$, but for simplicity, we restrict our study to real Banach spaces. 

Having a trivial group of isometries is not preserved under linear isomorphism. Indeed, $\isom(X)$ might change drastically after a renorming: Bellenot showed in \cite{BellenotRenorming} that every separable real Banach space can be renormed in such a way that its isometry group is trivial. However, it is an isometric property: if $X$ and $Y$ are isometrically isomorphic Banach spaces, then $\isom(X) = \{\pm Id\}$ if and only if $\isom(Y) = \{\pm Id\}$. Bellenot also proved in \cite{BellenotJames} that this is the case of the James space $J_2$. This result was improved in \cite{dongNiJames}, where generalized James spaces $J_p$, for $p>1$, are shown to have trivial group of isometries.

In contrast, the property of having a diagonal or standard group of isometries does not depend only on the isometric structure. To see this, notice that any Banach space $X$ with a Schauder basis $(e_n)$ can be identified with the sequence space
$$\left\{(x_n) \in \mathbb{R}^{\mathbb{N}}: \sum_{n=1}^\infty x_n e_n \text{ converges in }X\right\},$$ 
equipped with the norm $\Vert (x_n)\Vert = \Vert \sum_{n=1}^\infty x_n e_n \Vert$. Recall that $c_0$ and $\ell_p$, for $p\neq 2$, are known to have standard group of isometries under their usual representation, which is induced by their canonical basis. However, this situation changes when we represent $c_0$ using the summing basis $(s_n)$, where $s_n=\sum_{i=1}^n e_i$. The operator $T: c_0 \rightarrow c_0$ defined by $T(e_1)=-e_1$ and $T(e_n)=e_n$ for $n \geq 2$ is an isometry of the form \eqref{definition diagonal group isom} under the usual (canonical) representation. On the other hand, $T(s_2) = -e_1 + e_2 \neq \pm s_n$ for every $n \in \mathbb{N}$, so that $c_0$ fails to have standard group of isometries if we use the summing basis to represent the space as sequence space. Another example comes from the isometry $S((x_n)) = (\sum_{i=n}^\infty x_i)$ mapping $\ell_1$ onto the generalized James space $J_1$. The operador $T: \ell_1 \rightarrow \ell_1$ defined by $T(e_1)=-e_1$ and $T(e_n)=e_n$ for $n \geq 2$ is an isometry of $\ell_1$, so that $S \circ T \circ S^{-1}$ is an isometry of $J_1$ clearly not satisfying \eqref{definition diagonal group isom}. Despite this dependence on representation, we will omit the basis or representation when referring to the isometry group as diagonal or standard, unless they are not the canonical ones. 

In the middle ground are spaces whose isometry group is not trivial, but which exhibit more rigidity than those of the classical sequence spaces. In this class lie the generalized Schreier spaces and their $p$-convexifications $X_{\mathcal{S}_\alpha}^p$, for $p \geq 1$, which have diagonal group of isometries, see \cite{antunesBeanlandChuGeometry, AdBrechFerenczi, fakhoury2023isometries}. Further intermediate cases include, for example, the space $c$ of convergent sequences of real numbers. Its isometries satisfy \eqref{definition standard group isom}, but the sequence of signs must be, moreover, eventually constant, see \cite{AntunesBeanlandsurvey}. When replacing the generalized Schreier families $\mathcal{S}_\alpha$ by more general hereditary families $\mathcal{F}$ to get $X^p_\mathcal{F}$, not every permutation induce an isometry, see \cite{AdBrechFerenczi, BrechPina, fakhoury2023isometries}. All of these works rely heavily on two main ingredients: a characterization of the extreme points in the dual unit ball, and the fact that isometries must preserve disjointness of supports. For example, the extreme points of the James space $J_2$ were characterized in \cite{BellenotJames} and, more recently, in \cite{ArgyrosGonzalez}. In the case of the Schreier space $X_\mathcal{S}$, the extreme points were characterized in \cite{AdBrechFerenczi}. 

Our main result, Theorem \ref{teo: extreme points in BV1*}, characterizes the extreme points of the dual unit ball of the James-Schreier space $V_1$. The \emph{James-Schreier} spaces $V_p$, for $p\geq 1$, were introduced and studied in \cite{JSI,JSII,JSremarks}. These spaces amalgamate the James space $J_p$ and the Schreier space in a similar way as the sequence space $\ell_p$ and the Schreier space are used to define the combinatorial Banach spaces $X_\mathcal{F}^p$. The James-Schreier spaces were constructed to produce a new example of a Banach algebra with bounded approximate identity. The authors prove that, despite having quite different Banach space properties, almost all results regarding the James space as a Banach algebra carry over to the James-Schreier spaces. For instance, $V_p$ and $J_p$ are weakly amenable but not amenable (see \cite{JSII}), $V_p$ is $c_0$-saturated, $J_p$ is $\ell_p$-saturated and James spaces are totally incomparable with both the Schreier and the James-Schreier spaces. While each James-Schreier space contains a complemented copy of its respective Schreier space, it does not embed into any Schreier space (see \cite{JSI}). To prove Theorem~\ref{teo: extreme points in BV1*}, we use the fact that the standard Schauder basis of $V_1$ is shrinking (see \cite{JSremarks}) along with an analysis of the unit ball of $V_1$ to characterize the extreme points of the dual unit ball $B_{V_1^*}$ in terms of their support and the values of the coordinates therein. We then use Theorem~\ref{teo: extreme points in BV1*} to prove Theorem~\ref{theo: V1 has trivial group of isometries}, which states that $V_1$ has a trivial group of isometries. This shows that, despite being built upon the Schreier family and the norm structure of the James space, the isometric structure of the James-Schreier space $V_1$ does not resemble that of either of the spaces amalgamated. This is not surprising, since results from \cite{JSremarks} show that the Schreier space $X_\mathcal{S}$ and the James-Schreier space $V_1$ have very different isomorphic structures.

In the study of linear surjective isometries between Banach spaces, one can take a more general perspective and seek for Banach-Stone-type results within a class of Banach spaces: knowing that two of its Banach spaces are isometric, what can be said about their common structure? The classical Banach-Stone Theorem states that two Banach spaces of continuous functions $C(K)$ and $C(L)$ are isometrically isomorphic if and only if $K$ and $L$ are homeomorphic. In the class of $\ell_p$ spaces, $\ell_p$ and $\ell_q$ are isometrically isomorphic if, and only if, $p=q$. A similar result follows from results from \cite{JSremarks} in the context of the $p$-convexification of Schreier spaces $X^p_{\mathcal{S}}$ and James-Schreier spaces $V_p$. In \cite{BrechPina} and \cite{fakhoury2023isometries}, Banach-Stone-type theorems were obtained in the context of combinatorial spaces $X_{\mathcal{F}}$ and their $p$-convexifications $X^p_{\mathcal{F}}$, respectively. 

In addition to our main results about the James-Schreier space, we analyze the isometries between Lorentz sequence spaces. The Lorentz sequence space $d(w,1)$, determined by a decreasing weight sequence $w$, is a classical sequence space having a unique (up to equivalence) symmetric Schauder basis (see \cite[Section 4.e]{LindenstraussTzafriri1}), and some of its geometric properties were studied in \cite{DentingPoints,CarothersIsometrisLorentz,ciesielskisequencelorentzstructure}. The extreme points of the unit ball of Lorentz spaces were characterized in \cite[Theorem 2.6]{AnnaLeeExtreme}. We use this characterization to prove Theorem \ref{theo: iso dv,dw are induced by permutation and scalar}, which states that two Lorentz spaces determined by strictly decreasing sequences of weights are isometric if and only if the sequences of weights are proportional. This can be seen as a Banach-Stone-type theorem, which agrees with analogous results obtained in \cite{CarothersIsometrisLorentz} for the Lorentz function space $L_{w,p}$. As consequences, we get that the isometry group of both the Lorentz space $d(w,1)$ and its predual $d_*(w,1)$ are standard, see Theorem \ref{theo: isometria standard predual lorentz}. These facts are probably well-known, as they follow from results in \cite{BravermanSemenov}. 

The paper is divided into two independent and self-contained parts. The first part is devoted to the results related to the James-Schreier space and consists of Sections \ref{sec V1 space}, \ref{sec extreme points v1*} and \ref{sec isometries V1}. In Section \ref{sec V1 space} we introduce the James-Schreier space $V_1$, an amalgamation of two classical Banach sequence spaces originally introduced by Bird and Laustsen in \cite{JSI}, and we present several technical lemmas concerning the evaluation of coordinates of functionals in the dual unit ball $B_{V_1^*}$. Section \ref{sec extreme points v1*} provides the main result (Theorem \ref{teo: extreme points in BV1*}), which charaterizes the extreme points of $B_{V_1^*}$, while its implications for the isometries are proved in Section \ref{sec isometries V1}. The second part, Section \ref{sec lorentz space}, is devoted to proving a Banach-Stone-type result for the Lorentz sequence spaces. 

We will adopt the following notation for Banach space theory: given a Banach space $X$, $X^*$ denotes its dual space, $B_X$ the closed unit ball of $X$, $S_X$ the unit sphere of $X$ and $\ext(B_X)$ the set of all extreme points in $B_X$, that is, the points $x\in B_X$ such that whenever $x = \tfrac{\alpha+\beta}{2}$ with $\alpha,\beta\in B_X$, then $x=\alpha=\beta$. Given a Banach space $X$ with a Schauder basis $(e_n)$ and $x =\sum_{i=1}^{\infty}\lambda_ie_i\in X$, the support of $x$ with respect to $(e_n)$ is the set $\supp x = \{i \in \mathbb{N}: \lambda_i \neq 0\}$. We also denote $\lambda_i$ by $x_i$ or $e_i^*(x)$. When $(e_n)_n$ is a shrinking basis of $X$, the sequence $(e_n^*)_n$ is a basis for $X^*$ and we denote the coordinates of a vector $x^*\in X^*$ by $x^*_i$ or $x^*(e_i)$. For a standard reference on Banach space theory, we refer the reader to \cite{megginsonbook}. $\mathbb{N}$ denotes the set of strictly positive integers. We refer to the book \cite{FlemingJamison} and the survey \cite{AntunesBeanlandsurvey} for detailed presentations on the study of isometries over function spaces and over sequence spaces, respectively.

\section{The James-Schreier space \texorpdfstring{$V_1$}{V1} and its dual}\label{sec V1 space}

The James-Schreier spaces $V_p$, for $p\geq 1$, were introduced and studied in \cite{JSI,JSII,JSremarks}. In this work we will focus on the James-Schreier space $V_1$. Let us recall its definition. 

\begin{defin}
    We say that a subset $A\subseteq \N$ is permissible if $2\leq |A| \leq \min A +1$. Denote by $\per$ the family of all permissible sets and $\permax$ the permissible sets that are maximal with respect to inclusion.

The definition of the Schreier family varies from reference to reference, but we can think of permissible sets as corresponding to those elements of the Schreier family which are not singletons.  
    It will be convenient to enumerate any nonempty subset $F$ of some permissible set $A$ as $F=\{n_1< n_2<\cdots < n_{k+1}\}$, where $0 \leq k\leq n_1$. $F$ is a singleton when $k=0$; it is an element of $\per$ when $1 \leq k\leq n_1$; and $F$ is an element of $\permax$ when $k=n_1$.
    
Given a sequence $x= (x_n)$ of real numbers, denoting $\nu(x,A) \doteq \sum_{j=1}^k |x_{n_j} - x_{n_{j+1}}|$ for each $A\in \per$, we define:

\[
\| x\|_{W_1} = \sup_{A\in \per} \nu(x,A) \leq \infty \quad \text{and}\quad W_1\doteq \{x\in c_0: \|x\|_{W_1}<\infty\}.
\]

It follows that $(W_1, \|\cdot\|_{W_1})$ is a Banach space. Moreover, if 
$(e_n)$ is the sequence of standard unit vectors in $c_{00}$, then each $e_n$ clearly belongs to $W_1$, so that we can consider its closed subspace $V_1\doteq \overline{\spn}\{e_n: n\in \N\} \subseteq W_1$, which is called the James-Schreier space (for $p=1$).
\end{defin}

From now on, we will omit the subscript $W_1$ in the norm of $V_1$. The following facts about the James-Schreier space hold: 
in \cite{JSI} the authors prove that $(e_n)$ is a Schauder basis of $V_1$ which is not unconditional and that $V_1$ is $c_0$-saturated. Moreover, it is proved in \cite{JSremarks} that the basis $(e_n)$ is shrinking, so that the sequence of biorthogonal functionals $(e_n^*)$ is a Schauder basis of $V_1^*$.

Our main purpose is to describe the group of surjective linear isometries of $V_1$, $\isom(V_1)$, in terms of permutations of the basis $(e_n)$ and possible changes of signs. To achieve our goal, we will characterize the extreme points of dual unit ball $B_{V_1^*}$ and describe $\isom(V_1^*)$. 

\begin{defin}
We say that a functional $x^*\in B_{V_1^*}$ is compatible if there are $k\geq 0$ and $F=\{n_1<\ldots<n_{k+1}\} \subseteq A\in \per$ such that:
$$x^* = \pm \left(\sum_{i=1}^{k+1} (-1)^i\lambda_i e_{n_i}^* \right) \text{ where }
\left\{\begin{array}{l}
\lambda_1=1, \\
\lambda_i=2 \text{ if } 1<i<k+1,\\
\lambda_{k+1} \in \{1,2\} \text{ and }\lambda_{k+1} = 1 \text{ if }F \in \permax.
\end{array}\right.$$
\end{defin}

Notice that if $x^* = \sum_{i=1}^\infty x^*_i e_i^*$ is a compatible functional, then any partial sum $\sum_{i=1}^m x^*_i e_i^*$ is also a compatible functional, if $m \geq \min \supp x^*$. 

The following example will be helpful to estimate the norm of vectors:
\begin{example} \label{exe: f-a-theta}
    Given $A=\{n_1,n_2\ldots,n_{k+1}\}\in\per$ and a sequence $\vec \theta = (\theta_1,\ldots,\theta_{k})$ of signs in $\R$, the functional $x^*_{A,\vec\theta}\doteq \sum\limits_{i=1}^k \theta_i(e_{n_i}^* - e_{n_{i+1}}^*) \in V_1^*$ is compatible. Moroever, every compatible functional $x^* = \pm \left(\sum_{i=1}^{k+1} (-1)^i\lambda_i e_{n_i}^* \right)$ with $k\geq 1$ and $\lambda_{k+1}=1$ is of the form $x^*_{A,\vec\theta}$ for some $A \in \per$ and some $\vec{\theta}$.

    In fact, it is straightforward to check that $x^*_{A,\vec\theta}\in B_{V_1^*}$, furthermore if one writes:
\[
x^*_{A,\vec\theta}=
\theta_1 e_{n_1}^* + \underbrace{(\theta_2 - \theta_1)}_{\text{0, 2 or -2}}e_{n_2}^* +\cdots+ \underbrace{(\theta_k - \theta_{k-1})}_{\text{0, 2 or -2}}e_{n_k}^* - \theta_k e_{n_{k+1}}^*,
\]
it is easy to verify that consecutive coordinates on $\supp\big(x^*_{A,\vec\theta}\big)$ have alternating signs, thus it is a compatible functional.
\end{example}

\begin{prop}\label{prop: M sign invariand and weak*closed}
    The set $M = \left\{ x^* \in V_{1}^*: x^* \text{ is compatible} \right\}\cup\{0\}$ is a norming set that is symmetric and weak$^*$-closed.
\end{prop}
\begin{proof}
Observe that $M$ is a norming set because in Example~\ref{exe: f-a-theta} we proved that each $x^*_{A,\vec\theta}$ is compatible. Hence, given any $x = (x_n)\in V_1$ and $A =\{n_1 < \dots < n_{k+1}\}\in \per$, let $\theta_i = \sgn (x_{n_i}-x_{n_{i+1}})$ and notice that $\nu(x,A)=x^*_{A,\vec\theta}(x)$. Moreover, it is clear that $M$ is symmetric. We shall prove that $M$ is weak$^*$-closed.

Let $x^*\in \overline{M}^{w^*}$ with $x^*\neq 0$, we will prove that $x^*$ is compatible. Because the canonical basis $(e_n)$ of $V_1$ is shrinking (see \cite{JSremarks}), we can write $x^* = \sum\limits_{i=1}^\infty x^*(e_i) e_i^*$ with $x^*(e_i)\rightarrow 0$. It follows that:
\[
\text{for all } i\in\N:\quad 
|x^*(e_i)| \in \overline{\left\{ |y^*(e_i)|: y^* \in M\right\}} = \{0,1,2\},
\]
thus, $\supp x^*$ is finite. Let $S\doteq \supp x^* = \{n_1,\ldots, n_{k+1}\}$, with $k\geq 0$.

We consider the basic weak$^*$-open neighbourhood of $x^*$ given by:
\[ U\doteq U\left(x^*, \{ e_1, \ldots , e_{n_{k+1}}\}, \frac{1}{2}\right) = \left\{ y^* \in V_1^*: \max\limits_{i\leq n_{k+1}} |(y^*-x^*)(e_i)|<\frac{1}{2} \right\}.
\]
Let $y^*\in M \cap U$ and since $y^*$ is compatible, we obtain that:
\[
\text{if } i\leq n_{k+1}: \begin{cases}
    i\in S \implies |y^*(e_i) - x^*(e_i)|<\frac{1}{2} \implies y^*(e_i) = x^*(e_i); \\
    i\notin S \implies |y^*(e_i)| = |y^*(e_i) - x^*(e_i)| <\frac{1}{2} \implies y^*(e_i) = x^*(e_i) = 0.
\end{cases}
\]

In particular, $S \subseteq \supp y^*$ and for every $i \in S$, $x^*(e_i) = y^*(e_i)$. Moreover, if $S \in \permax$, then $S=\supp y^*$ and, therefore, $|x^*(e_{n_{k+1}})| = 1$. It follows easily that, for $m = \max \supp y^*$, $x^* = \sum_{n=1}^{m} y_n^* e_n^*$, which is a compatible functional.
\end{proof}

We remember the following result, stated as Lemma 4 of \cite{AdBrechFerenczi}:

\begin{lemma}\label{lemma: norming set conv}
    Let $X$ be a Banach space over $\K$ and $N\subseteq B_{X^*}$ be a sign invariant norming set for $X$. Then $B_{X^*} = \overline{\conv(N)}^{w^*}$.
\end{lemma}

Notice that for real Banach spaces, being sign invariant is the same as being symmetric. As a consequence of Proposition~\ref{prop: M sign invariand and weak*closed}, Lemma~\ref{lemma: norming set conv} and Milman's Theorem (see \cite[Theorem 3.25]{RudinFA}), it follows that $\ext\big(B_{V_1^*}\big)\subseteq M$. 

In what follows, our main goal is to determine which elements of $M$ are extreme points of $B_{V_1^*}$. We present preliminary results which will be used in Section~\ref{sec extreme points v1*} to show that certain compatible functionals are indeed extreme points. Given a compatible functional $x^* = \sum\limits_{i=1}^{k+1} x_{n_i}^* e_{n_i}^*$, we start with $\alpha,\beta\in B_{V_1^*}$ such that $\alpha + \beta = 2x^*$ and show that, under certain assumptions,  $\alpha_j=\beta_j= x_j^*$ for every $j \in \mathbb{N}$. Evaluating these functionals on specific vectors in the unit ball of $X$ will be crucial for that purpose. 

Notice that if $x\in V_1$, then $x \in c_0$, so that
$$|x(i)| = \lim_{j \rightarrow \infty} \nu(x,\{i,j\}) \leq \sup_{A \in \mathbb{P}} \nu(x,A) = \Vert x \Vert.$$
Therefore, $\|x\|_\infty \leq \|x\|$, thus $B_{V_1}\subseteq [-1,1]^\N$.

Let us introduce some further notation. Given a set $A\subseteq \N$, let $\rchi_{A} \doteq \sum_{j\in A} e_j$ and $x^*|_A \doteq x^*|_{\spn\{e_i: \,i\in A\}}$,  when $x^*\in V_1^*$. We also use interval notation within $\mathbb{N}$, e.g. $[i,j]=\{k \in \mathbb{N}: i \leq k \leq j\}$ and $]i,+\infty[=\{k \in \mathbb{N}: i < k \}$.

We start by pointing out a collection of vectors which are in $B_{V_1}$:

\begin{lemma}\label{lemma: bola}
The following vectors are in the unit ball $B_{V_1}$:
\begin{enumerate}
    \item $e_i$, for $i \leq 2$;
    \item $\frac{e_i}{2}$, for $i\geq 3$;
    \item $e_1 + \tfrac{e_j}{2}$, for $j \geq 2$;
    \item $\tfrac{e_i\pm e_j}{2}$, for $i \leq 2 \leq j$;
    \item $\rchi_{[1,j]}$, for $j\geq 2$;
    \item $\frac{\rchi_{[i,j]}}{2}$, for $2 \leq i<j$.
\end{enumerate}
\end{lemma}

In what follows, we will use Lemma \ref{lemma: bola} several times without explicit mention. Any vector which is claimed to be in $B_{V_1}$ should be listed above. The following lemmas, despite being simple, exemplify the use of these vectors to compute certain coordinates.

\begin{lemma}\label{lemma: coordinates}
Given $\alpha, \beta \in B_{V_1^*}$, the following assertions hold:
\begin{enumerate}
    \item If $i \leq 2$ and $\alpha_i + \beta_i = \pm 2$, then $\alpha_i=\beta_i=\pm 1$.
    \item If $i > 2$, $\alpha_i+\beta_i=\pm2$ and $\alpha|_{[1,i[} = \beta|_{[1,i[} \equiv 0$, then $\alpha_i=\beta_i=\pm1$.   
    \item If $i \leq 2$, $\alpha_i + \beta_i = 0$ and $\alpha_j=\beta_j=\pm2$ for some $j>2$, then $\alpha_i=\beta_i=0$.
\end{enumerate}
\end{lemma}
\begin{proof}
In (1) and (2), we get that $|\alpha_i|, |\beta_i| \leq 1$ by evaluating $\alpha$ and $\beta$ on the unit vectors $e_i$ and $\rchi_{[1,i]}$, respectively. It follows that $\alpha_i=\beta_i=\pm1$. To prove (3), we evaluate $\alpha$ and $\beta$ on $\tfrac{e_i+ e_j}{2} \in B_{V_1}$, obtaining either $\left|\frac{\alpha_i}{2} +1\right|, \left|\frac{\beta_i}{2} +1\right|\leq 1$ or $\left|\frac{\alpha_i}{2} -1\right|, \left|\frac{\beta_i}{2} -1\right|\leq 1$, and both alternatives imply that $\alpha_i=\beta_i=0$. 
\end{proof}

\begin{lemma}\label{lemma: functionals coincide on ]n1,nk+1[}
    Let $x^* = \sum\limits_{i=1}^{k+1} x_{n_i}^* e_{n_i}^*$ be a compatible functional and $\alpha,\beta \in B_{V_1^*}$ be such that $\alpha+\beta = 2x^*$. 
    \begin{enumerate}
        \item If $|\supp x^*|>2$, then $\alpha|_{]n_1,n_{k+1}[} = \beta|_{]n_1,n_{k+1}[} = x^*|_{]n_1,n_{k+1}[}$. 
        \item Moreover, if also $|x_{n_{k+1}}^*| = 2$, then $\alpha|_{]n_1,+\infty[} = \beta|_{]n_1,+\infty[} = x^*|_{]n_1,+\infty[}$. 
    \end{enumerate}
\end{lemma}
\begin{proof}
    We may assume without loss of generality that $x_{n_1}^* = 1$. Notice that $|\supp x^*|>2$ implies that $n_1\geq 2$. For $1<i \leq k+1$, since $\frac{e_{n_i}}{2} \in B_{V_1}$, the evaluation of $\alpha$ and $\beta$ yields
$|\alpha_{n_i}|, |\beta_{n_i}| \leq 2$. If $i < k+1$, since $\alpha_{n_i}+\beta_{n_i} = 2 x_{n_i}^* =\pm4$, we get that $\alpha_{n_i} = \beta_{n_i} = x_{n_i}^*$. 

Let $1< i < k +1$ and $n_i< j< n_{i+1}$ be such that for every $n_i < m < j$, $\alpha_m=\beta_m=0$. Since $\Vert \tfrac{\rchi_{[n_i,j]}}{2} \Vert \leq 1$, we can evaluate $\alpha$ and $\beta$ to get 
\[
\text{either}\quad \left| 1+\frac{\alpha_j}{2}\right|,
\left| 1+ \frac{\beta_j}{2}\right| \leq 1
\quad\text{ or}\quad
\left| 1- \frac{\alpha_j}{2}\right|, \left| 1- \frac{\beta_j}{2}\right| \leq 1,
\]
Both imply that $\alpha_j = \beta_j= 0$. If $n_1< j< n_{2}$, a similar inductive argument applies to show that $\alpha_j = \beta_j= 0$, by evaluating $\alpha$ and $\beta$ on $\tfrac{\rchi_{[j,n_2]}}{2} \in B_{V_1}$.

Finally, if $|x^*_{n_{k+1}}|=2$, we use vectors $\frac{e_{n_{k+1}}}{2}$ and $\tfrac{\rchi_{[n_{k+1},j]}}{2}$ to show that $\alpha_{n_{k+1}} = \beta_{n_{k+1}} = x_{n_{k+1}}^*$ and $\alpha_j=\beta_j=0$ for $j> n_{k+1}$, which concludes the proof. 
\end{proof}

For the remaining coordinates, it will be necessary to compute the norm of other vectors of $V_1$. To tackle the initial segment $[1,n_1[$ and the tail $]n_{k+1},+\infty[$, we need the following remark:

\begin{rem}\label{extensions}
Given a finitely supported vector $x\in V_1$, let $n= \min \supp x$. For $i \leq n -1$ and $t = x(n)$, we have that
$$\| t\chi_{[i,n-1]} + x\| \leq \|x\|.$$
Similarly, if $n= \max \supp x$, for $j \geq n+1$ and $t = x(n)$, we have that
$$\| x + t\chi_{[n+1,j]}\| \leq \|x\|.$$
\end{rem}

The points defined on Lemma \ref{lemma: vectorsyz v2} will be used to determine the coordinates $n_1$ and $n_{k+1}$. 

\begin{lemma}\label{lemma: vectorsyz v2}
Given $F = \{n_1,\ldots, n_{m}\}$ and $\vec{\theta} = (\theta_{n_i})_{i=1}^{m} \in \{-1,1\}^F$, consider the following vector:
$$y_{\vec{\theta}} \doteq \sum_{i=1}^{m} \theta_{n_i} e_{n_i}.$$
\begin{enumerate}
    \item If $m = n_1$, then $\Vert y_{\vec{\theta}}\Vert\leq 2m-1$.
    \item If $m = n_1 +1$, then $\Vert y_{\vec{\theta}}\Vert\leq 2(m-1)$.
    \item If $m \geq n_1$ and $\theta_{n_1} = \theta_{n_2}$, then $\Vert y_{\vec{\theta}} \Vert\leq 2(m-1)$.
\end{enumerate}
\end{lemma}

\begin{proof}
    For fixed $F$ and $\vec{\theta}$, notice that the values of the coordinates of $y_{\vec{\theta}}$ have at most $m-1$ jumps of height $2$, between two coordinates in the support, and two jumps of height $1$, associated to the first and last nonzero coordinates. 
    
    If $m=n_1$, then a permissible set $A$ containing $F$ may contain only one further element, so that $\nu(y_{\vec{\theta}},A) \leq 2(m-1)+1 = 2m-1$, while a permissible set which does not contain $F$ may count the two jumps of height $1$, but necessarily misses at least one jump of height $2$, so that $\nu(y_{\vec{\theta}},A) \leq 2(m-2)+2 = 2(m-1)$. Hence, $\Vert y_{\vec{\theta}}\Vert \leq 2m-1$.
        
    If $m=n_1+1$, then a permissible set $A$ containing $F$ equals $F$, since $F$ is maximal. Then it misses the first and last jumps and this yields $\nu(y_{\vec{\theta}},A)\leq 2(m-1)$. A permissible set missing at least one element of $F$ may count the two jumps of height $1$, but necessarily misses at least one jump of height $2$, so that $\nu(y_{\vec{\theta}},A) \leq 2(m-2)+2 = 2(m-1)$. Therefore, $\Vert y_{\vec{\theta}}\Vert \leq 2(m-1)$.
    
    Finally, if $m\geq n_1$ and $\theta_{n_1} = \theta_{n_2}$, the coordinates of $y_{\vec{\theta}}$ 
    have at most $m-2$ jumps of height $2$ - from $\theta_{n_i}$ to $\theta_{n_{i+1}}$ for $2 \leq i \leq m$ - and $4$ jumps of height $1$: from $0$ to $\theta_{n_1}$, then back to $0$ and back to $\theta_{n_2}=\theta_{n_1}$, and finally a last jump from $\theta_{n_{m}}$ to $0$.    

    Thus, a permissible set $A$ containing $n_1$ hast at most $n_1+1\leq m+1$ elements. Hence, either $A$ contains $F$ and one more element, which makes it consider all steps of height $2$ and at most $2$ more jumps, which ensures $\nu(y_{\vec{\theta}},A) \leq 2(m-2)+2= 2(m-1)$; otherwise it misses at least one jump of height $2$ and may consider all $4$ jumps of height $1$, which yields $\nu(y_{\vec{\theta}},A) \leq 2(m-3)+4= 2(m-1)$. Finally, $A$ does not contain $n_1$, then it necessarily does not count two jumps of height $1$, so that $\nu(y_{\vec{\theta}},A) \leq 2(m-2) +2 = 2(m-1)$. This concludes the proof. 
\end{proof}

\section{Extreme points of \texorpdfstring{$B_{{V_1}^*}$}{BV1*}}\label{sec extreme points v1*}

In this section we will determine what are the compatible functionals $x^* = \sum\limits_{i=1}^{k+1} x_{n_i}^*e_{n_i}^* \in M$ that lie on $\ext\big(B_{V_1^*}\big)$. To ease the reading, the main theorem of this section, Theorem~\ref{teo: extreme points in BV1*}, is split on smaller propositions that will deal with different types of compatible functionals, highlighting where the points defined on the previous section are used.

The structure of the arguments to show that a given $x^*\in M$ is an extreme point of $B_{V_1^*}$ will follow the pattern: we suppose that $\alpha,\beta\in B_{V_1^*}$ are such that $\alpha+ \beta = 2 x^*$ and argue that $\alpha=\beta$. Notice that $x^*$ is an extreme point if, and only if, $-x^*$ is an extreme point. Starting from a compatible functional $x^*$ which we want to decide if it is an extreme point of $B_{{V_1}^*}$ or not, this allows us to suppose, without loss of generality, that the first coordinate on the support of $x^*$ is $x^*_{n_1}=1$.

\begin{prop} \label{Prop: extV1* singleton support}
$\pm e_i^* \in \ext\big(B_{V_1^*}\big)$ $\iff$ $i\in\{1,2\}$.    
\end{prop}

\begin{proof}
Without loss of generality we consider only $e_i^*$.

$(\Rightarrow)$ For $i\geq 3$, fix $0<\delta<1$ and consider:
    \[
    \alpha \doteq \delta e_{i-1}^* + (1-\delta)e_{i}^* 
    \qquad \text{and}\qquad
    \beta \doteq -\delta e_{i-1}^* + (1+\delta)e_{i}^*.
    \]
Because $\alpha+\beta =2e_i^*$ with $\alpha\neq \beta$, to prove that $e_i^*$ is not extreme, we need to show that $\alpha,\beta\in B_{V_1^*}$. Indeed, if $x\in B_{V_1}$, then:
\[
|\alpha(x)| = |\delta x_{i-1} + (1-\delta)x_{i}|\leq \delta \|x\| + (1-\delta)\|x\| = \|x\| \leq 1
\]
and
\[
|\beta(x)|= |-\delta x_{i-1} + (1+\delta)x_{i}|
        <  |x_i - x_{i-1}| + |x_i|\]
\[        
        = \lim_{n\to\infty} |x_i - x_{i-1}| + |x_i - x_n|  = \lim_{n\to\infty} \underbrace{\nu(x, \{i-1,i,n\})}_{\leq 1} \leq 1.
\]
Thus, $e_i^* \notin \ext\big(B_{V_1^*}\big)$ for $i\geq 3$. 

$(\Leftarrow)$ Conversely, for $i=1$, let $\alpha, \beta\in B_{V_1^*}$ be such that $\alpha+\beta = 2e_1^*$. From (1) of Lemma \ref{lemma: coordinates} we conclude that $\alpha_1 = \beta_1 = 1$. For each $j\in \N$, we evaluate $\alpha$ and $\beta$ on $\chi_{[1,j]}\in B_{V_1}$ to obtain, by induction on $j\geq 2$:
\[
|1 + \alpha_j| = \left|\alpha(\chi_{[1,j]}) \right| \leq 1
\qquad \text{and}\qquad
 |1 + \beta_j| = \left|\beta(\chi_{[1,j]}) \right|\leq 1.
\]
It follows that $\alpha_j, \beta_j \leq 0$ and since $\alpha_j+\beta_j=0$, we conclude that $\alpha_j = \beta_j=0$. Hence, $\alpha = \beta = e_1^*$ and $e_1^* \in \ext\big(B_{V_1^*}\big)$. The proof for $i=2$ follows analogously.
\end{proof}

As we mentioned in the previous section, we will use the points defined on Lemma~\ref{lemma: bola}, Remark~\ref{extensions} and Lemma~\ref{lemma: vectorsyz v2} to determine whether or not $x^*$ is an extreme point when its last coordinate is such that $|x_{n_{k+1}}| = 1$. We present:

\begin{prop} \label{Prop: extV1* |x_n{k+1}|=1}
Let $x^* = \sum\limits_{i=1}^{k+1} x^*_{n_i}e_{n_i}^*\in M$ be such that $k \geq 1$ and  $|x_{n_{k+1}}^*|=1$. Then:
\[x^*\in \ext(B_{V_1^*}) \iff \supp x^* \in \permax.\]
\end{prop}

\begin{proof}
$(\Rightarrow)$
Suppose that $\supp x^*\notin \permax$. Let $l>n_{k+1}$ and define:
\[
   \alpha \doteq x^* + (e_{n_{k+1}}^* - e_l^*) 
   \qquad \text{and}\qquad 
   \beta \doteq  x^* - (e_{n_{k+1}}^* - e_l^*).
\]
Since $\{n_1, \dots, n_{k+1}, l \} \in \per$, it follows from Example~\ref{exe: f-a-theta} that $\alpha$ and $\beta$ are compatible funcionals. In particular, $\alpha,\beta\in B_{V_1^*}$.  Since $\alpha + \beta = 2x^*$, we get that $x^* \notin \ext\big(B_{V_1^*}\big)$. 

$(\Leftarrow)$ Suppose that $\supp x^*\in\permax$, let $\alpha,\beta\in B_{V_1^*}$ such that $\alpha+\beta = 2x^*$ and assume, without loss of generality, that $x_{n_1}^* = 1$. To prove that $x^*$ is an extreme point, we will divide the argument in the following two cases:

\underline{\textbf{Case 1:}} $|\supp x^*| = 2$.

In this case, $x^* = e_1^* - e_j^*$ for some $j > 1$. Then it follows from (1) of Lemma \ref{lemma: coordinates} that $\alpha_1 = \beta_1 =1$. Also, since $\dfrac{e_1\pm e_j}{2}\in B_{V_1}$, we have:
\[
\left|\dfrac{1}{2} + \dfrac{\alpha_j}{2}\right|  = \left|\alpha\left( \dfrac{e_1+e_j}{2}\right)\right| \leq 1 
\quad \text{ and }\quad
\left|\dfrac{1}{2} - \dfrac{\alpha_j}{2}\right| = \left|\alpha\left( \dfrac{e_1-e_j}{2}\right)\right|  \leq 1
\]
and we conclude that $|\alpha_j|\leq 1$. Similarly, $|\beta_j|\leq 1$. Since $\alpha_j+\beta_j=2x^*_j = -2$, we get that $\alpha_j = \beta_j = -1 = x_j^*$. 
Finally, to show that $\alpha=\beta$, let $i>1$, $i \neq j$. By evaluating $\alpha$ and $\beta$ on $e_1 + \dfrac{e_i}{2} \in B_{V_1}$, we conclude that $\alpha_i, \beta_i \leq 0$. Since $\alpha_i+\beta_i=0$, we get that $\alpha_i=\beta_i=0$. Thus $\alpha=\beta=x^*$ and $x^*$ is extreme.

\underline{\textbf{Case 2:}} $|\supp x^*| > 2$.

Write $x^* = \sum\limits_{i=1}^{k+1} x_{n_i}^* e_{n_i}^*$ and observe that $\alpha|_{]n_1,n_{k+1}[}=\beta|_{]n_1,n_{k+1}[}=x^*|_{]n_1,n_{k+1}[}$, by Lemma~\ref{lemma: functionals coincide on ]n1,nk+1[}.

\textbf{Claim 1:} $\alpha_{n_1} =\beta_{n_1} = x_{n_1} =1$.

\begin{proof}[Proof of Claim 1.]
Consider the vectors 
$$y = e_{n_1} + \sum_{i=2}^{k} (-1)^{i-1} e_{n_i} \quad \text{and} \quad \bar{y} = e_{n_1}  - \sum_{i=2}^{k} (-1)^{i-1} e_{n_i}.$$
Since $\supp x^* = \{n_1, \ldots,n_{k+1}\}\in \permax$, we have that $k=n_1$. Hence, $m \doteq |\supp y|=|\supp \bar{y}| = k = n_1$. It follows from (1) of Lemma \ref{lemma: vectorsyz v2} that $y$ and $\bar{y}$ have norm at most $2k-1$.

The evaluation of $\alpha$ on them yields:
\begin{equation*} 
\alpha_{n_1}+2(k-1) = (2k-1)-(1-\alpha_{n_1}),
\end{equation*}
and
\begin{equation*}
\alpha_{n_1}-2(k-1) = -(2k-1)+(1+\alpha_{n_1}) 
\end{equation*}
respectively. Using that $\alpha \in B_{V_1}$, we get that $1+\alpha_{n_1} \geq 0$ and $1-\alpha_{n_1} \geq 0$, which implies that $|\alpha_{n_1}|\leq 1$. Similarly, $|\beta_{n_1}|\leq 1$. Since $\alpha_{n_1} + \beta_{n_1} = 2x^*_{n_1}=2$, we conclude that $\alpha_{n_1} = \beta_{n_1} = 1 = x^*_{n_1}$.
\end{proof}

\textbf{Claim 2:} $\alpha|_{[1,n_1[}=\beta|_{[1,n_1[} = x^*|_{[1,n_1[} \equiv 0$.

 \begin{proof}[Proof of Claim 2.] 
Given $1 \leq j < n_1$, suppose that for all $j < m < n_1$ we have that $\alpha_m=\beta_m =0$. The evaluation of $\alpha$ and $\beta$ on $\chi_{[j,n_1-1]} + y$, by Remark \ref{extensions} and Lemma \ref{lemma: vectorsyz v2} gives us that:
\[
   \alpha_j + 2k-1\leq \Vert \alpha\Vert \Vert \chi_{[j,n_1-1]} + y \Vert \leq \Vert y \Vert \leq 2k-1
\]
and 
\[
   \beta_j +2k-1 \leq \Vert \beta\Vert \Vert \chi_{[j,n_1-1]} + y\Vert \leq \Vert y \Vert \leq 2k-1, 
\]
so that $\alpha_j, \beta_j \leq 0$. Since $\alpha_j + \beta_j = 2x^*_j=0$, this implies that $\alpha_j = \beta_j = 0$.
\end{proof}

\textbf{Claim 3:} $\alpha_{n_{k+1}} =\beta_{n_{k+1}} = x_{n_{k+1}} =1$.

\begin{proof}[Proof of Claim 3.] For the coordinates above $n_{k+1}$, let us consider the vectors 
$$z = \sum_{i=1}^{k} (-1)^{i-1} e_{n_i} + e_{n_{k+1}} \quad \text{ and } \quad \bar{z} =  \sum_{i=1}^{k} (-1)^{i-1} e_{n_i} - e_{n_{k+1}}$$
Since $\supp x^* = \{n_1, \ldots,n_{k+1}\}\in \permax$, we have that $k=n_1$. It follows from (2) of Lemma \ref{lemma: vectorsyz v2} that $z$ and $\bar{z}$ have norm at most $2k$.

The evaluation of $\alpha$ on these vectors yields:
$$
    2k-1 + \alpha_{n_{k+1}} =2k- (1-\alpha_{n_{k+1}}) \,\, \text{ and } \,\, 2k -1 -\alpha_{n_{k+1}} = 2k -(1+ \alpha_{n_{k+1}}),
$$
respectively. Since $\alpha$ has norm $1$, it follows that $|\alpha_{n_{k+1}}|\leq 1$. Similarly, $|\beta_{n_{k+1}}|\leq 1$. Thus, $\alpha_{n_{k+1}} = \beta_{n_{k+1}} = x^*_{n_{k+1}}$.
\end{proof}

\textbf{\underline{Claim 4:}} $\alpha|_{]n_{k+1},+\infty[}=\beta|_{]n_{k+1},+\infty[} = x^*|_{]n_{k+1},+\infty[} \equiv 0$.

\begin{proof}[Proof of Claim 4.] Let $j>n_{k+1}$ be such that $\alpha_m=\beta_m=0$ for every $n_{k+1} < m < j$. Suppose $x^*_{n_{k+1}}=1$. Notice that, by Remark \ref{extensions} and (2) of Lemma \ref{lemma: vectorsyz v2}, 
$$\left \Vert z + \chi_{[n_{k+1}+1,j]}
\right\Vert=
\left \Vert z \right\Vert\leq 2k.$$
Also, 
$$\alpha \left(z + \chi_{[n_{k+1}+1,j]}\right)
   = (1+ 2(k-1) + 1) + \alpha_j 
    = 2k + \alpha_j.$$
This determines the sign of $\alpha_j$.
Similarly, the same inequality holds for $\beta_j$, so that $\sgn(\alpha_j)=\sgn(\beta_j)$ and, since $\alpha_j+\beta_j = 2x^*_j =0$, we get that $\alpha_j=\beta_j=0$. In case $x^*_{n_{k+1}}=-1$, we replace $z + \chi_{[n_{k+1}+1,j]}$ by $\bar{z} - \chi_{[n_{k+1}+1,j]}$ and the same argument guarantees that $\alpha_j=\beta_j=0$. \end{proof}
Thus, $x^* = \alpha=\beta$ and $x^*\in\ext\big(B_{V_1^*}\big)$.
\end{proof}

The next proposition deals with the case where $x^*$ is a compatible functional and its last non-zero coordinate is $\pm2$. We will use different techniques to decide whether or not $x^*$ is an extreme point, depending on the cardinality of $\supp x^*$.

\begin{prop}
Let $x^* = \sum\limits_{i=1}^{k+1} x_{n_i}^* e_{n_i}^* \in M$ be such that $|x_{n_{k+1}}^*| =2$. Then:
$$x^* \in \ext\big( B_{V_1^*}\big) \iff n_1-1\leq |\supp x^*| \leq n_1.$$
\end{prop}
\begin{proof}
Notice that $|x_{n_{k+1}}^*| =2$ implies $k \geq 1$. For both implications we can assume, without loss of generality, that $x_{n_1}^* = 1$. 

$(\Rightarrow)$  Suppose that $|\supp x^*| <n_1-1$ and notice that $A\doteq\{n_1-1, n_1, n_2, \ldots, n_{k+1}\}\in \per\setminus\permax$. 

Let $0<\delta<1$ and define:
\[\alpha\doteq \delta e_{n_1-1}^* + (1-\delta)e_{n_1}^* + \sum_{i=2}^{k+1} x_{n_i}^*e_{n_i}^*
\qquad\text{ and }\qquad
\beta\doteq -\delta e_{n_1-1}^* + (1+\delta)e_{n_1}^* + \sum_{i=2}^{k+1} x_{n_i}^*e_{n_i}^*
\]
Clearly, $\alpha+\beta = 2x^*$. To conclude the proof we need to show that $\alpha,\beta\in B_{V_1^*}$. Given $y\in B_{V_1}$, then:
\begin{align*}
    |\alpha(y)|  &= | \delta y_{n_1-1} + (1-\delta)y_{n_1} -2 y_{n_2} + \cdots (-1)^k 2 y_{n_{k+1}}|\\ 
    &= |\delta(y_{n_1-1} - y_{n_1}) + (y_{n_1} - y_{n_2}) -(y_{n_2}-y_{n_3}) +\cdots \\
    &\hphantom{= \quad}
            + (-1)^{k-1}(y_{n_k} - y_{n_{k+1}}) + (-1)^k y_{n_{k+1}}| \\
            &\leq 
            |y_{n_1-1} - y_{n_1}| + |y_{n_1} - y_{n_2}| + \cdots +|y_{n_k} - y_{n_{k+1}}|+|y_{n_{k+1}}|
\end{align*}

  Notice that, for all $j> n_{k+1}$, $B_j \doteq \{n_1 -1,n_1,\ldots, n_{k+1},j\} \in \per$ and $\nu(y, B_j) \leq 1$. Thus:
\[
|y_{n_1-1} - y_{n_1}| + |y_{n_1} - y_{n_2}| + \cdots +|y_{n_k} - y_{n_{k+1}}|+|y_{n_{k+1}}| = \lim_{j\to\infty} \nu(y, B_j)  \leq 1
\]
and we conclude that $\alpha\in B_{V_1^*}$. In the same way, one can prove that $\beta\in B_{V_1^*}$. Therefore, $x^*\notin \ext\big( B_{V_1^*}\big)$.

$(\Leftarrow)$ Conversely, let $x^* = \sum\limits_{i=1}^{k+1} x_{n_i}^* e_{n_i}^* \in M$ with $k\geq 1$, $x_{n_1}^*=1$, $|x_{n_{k+1}}^*| =2$ and let $\alpha,\beta\in B_{V_1^*}$ be such that $\alpha+\beta = 2x^*$. We split the proof in the following cases:

\underline{\textbf{Case 1:}} $|\supp x^*| = n_1$.

 It follows directly from Lemma~\ref{lemma: functionals coincide on ]n1,nk+1[} that $\alpha|_{]n_1, +\infty[} = \beta|_{]n_1, +\infty[}  = x^*|_{]n_1, +\infty[}$. The proofs of the next two claims are similar to those of Claims 1 and 2 in the previous proposition.

\textbf{Claim 1:} $\alpha_{n_1} =\beta_{n_1} = x_{n_1} =1$.

\begin{proof}[Proof of Claim 1.] Consider the vectors
 $$y = e_{n_1} + \sum_{i=2}^{k+1} (-1)^{i-1} e_{n_i} \quad \text{ and } \quad \bar{y} = e_{n_1}  - \sum_{i=2}^{k+1} (-1)^{i-1} e_{n_i}.$$
Since $|\supp y| =|\supp \bar{y}| = k+1 =| \supp x^*| = n_1$, it follows from (1) of Lemma~\ref{lemma: vectorsyz v2}, that $\Vert y \Vert, \Vert \bar{y} \Vert \leq 2(k+1)-1=2k+1$. The evaluation of $\alpha$ and $\beta$ on these vectors yields $|\alpha_{n_1}|\leq 1$ and $|\beta_{n_1}|\leq 1$. Because $\alpha_{n_1}+\beta_{n_1} = 2$, we conclude that $\alpha_{n_1} = \beta_{n_1} = x_{n_1}^* = 1$.
\end{proof}

\textbf{Claim 2:} $\alpha|_{[1,n_1[}=\beta|_{[1,n_1[} = x^*|_{[1,n_1[} \equiv 0$.

\begin{proof}[Proof of Claim 2.] Assume that $1 \leq j < n_1$ is such that for every $j < m<n_1$, $\alpha_m=\beta_m=0$. By evaluating $\alpha$ and $\beta$ on $\chi_{[j,n_1-1]}+ \bar{y}$, using Remark \ref{extensions} and Lemma \ref{lemma: vectorsyz v2}, we conclude that $\alpha_j=\beta_j=0$. Hence, by induction, we get that $\alpha|_{[1,n_1[} = \beta|_{[1,n_1[}  = x^*|_{[1,n_1[}\equiv 0$.
\end{proof}

This ensures that $\alpha=\beta=x^*$, so that $x^*$ is an extreme point and Case 1 is finished.

  \underline{\textbf{Case 2:}} $|\supp x^*| = n_1-1$.

    Since $n_1 \geq 3$, from (3) of Lemma~\ref{lemma: coordinates} and Lemma~\ref{lemma: functionals coincide on ]n1,nk+1[}, it remains to prove the following claim:
    
    \textbf{Claim 3:} 
    $\alpha|_{[3,n_1]} = \beta|_{[3,n_1]} = x^*|_{[3,n_1]}$. 
    
    \begin{proof}[Proof of Claim 3.] If $n_1=3$, (2) of Lemma~\ref{lemma: coordinates}
    gives that $\alpha_{n_1} = \beta_{n_1} = x^*_{n_1}$, which concludes the proof.

    If $n_1 >3$, for each $3\leq l<n_1$, let
\begin{equation*}
    w_l = - e_l^* + \sum_{i=2}^{n_{k+1}} (-1)^{i-1} e_{n_i}^*.
\end{equation*} 
    Since $|\supp w_l| = k+1 = n_1-1$ and $\min \supp w_l = l \leq n_1-1$, it follows from (3) of Lemma  \ref{lemma: vectorsyz v2} that $\Vert w_l \Vert \leq 2k$. When we evaluate $\alpha$ and $\beta$ on $w_l$, we get that    
    \[
|- \alpha_l + 2k| \leq 2k \quad \text{ and } \quad |-\beta_l + 2k| \leq 2k.    \]
Hence, since $\alpha_l + \beta_l =  2x^*_l =0$, it follows that $\alpha_l = \beta_l =  x^*_l =0$. Now we can apply (2) of Lemma~\ref{lemma: coordinates} and conclude that $\alpha_{n_1}=\beta_{n_1}=x^*_{n_1}$.
\end{proof}
Therefore, $x^* \in \ext\big( B_{V_1^*}\big)$.
\end{proof}

The conjunction of all propositions in this section allows us to characterize the extreme points of $B_{V_1^*}$ as follows:

\begin{theo}\label{teo: extreme points in BV1*}
    An element  $x^* \in B_{V_1^*}$ is an extreme point if, and only if, $x^* = \sum\limits_{i=1}^{k+1} x_{n_i}^*e_{n_i}^*$ is a compatible functional satisfying one of the following conditions:

    \begin{enumerate}
        \item $x^*= \pm e_i^*$, for $i\in \{1,2\}$;
        \item $\supp x^* \in\permax$;
        \item $|x_{n_{k+1}}^*| = 2$ and $n_1-1\leq |\supp x^*| \leq n_1$.
    \end{enumerate}
\end{theo}

\section{Isometries of \texorpdfstring{$V_1$}{V1}}\label{sec isometries V1}

In this section we will use the set of all compatible functionals $M$, defined on Proposition~\ref{prop: M sign invariand and weak*closed}, and the characterization of the extreme points in $B_{V_1^*}$ given in Theorem~\ref{teo: extreme points in BV1*} to describe $\isom(V_1)$ in terms of signed permutations. 

In \cite{AdBrechFerenczi,BrechPina} it is proved that the only permutation that is allowed when one attempts to describe the group of isometries of the Schreier space is the identity, although there is no restriction over the sequence of signs. Our theorem shows that, in the case of surjective isometries on the James-Schereier space, we are even more restrict: the only surjective isometries in $\isom(V_1)$ are $\pm Id$. There are two crucial differences between the arguments used in the aforementioned papers and the study we have done so far: the first one is the fact that $\ext(B_{V_1^*})$ contains several extreme points with maximal support consisting of two points, namely the points of the form $e_1^* - e_l^*$, and the second one is that the elements of $\ext(B_{V_1^*})$ have alternating signs in their coordinates. The additional extreme points will play an important role to prove that $S(e_n^*) = \pm e_{\sigma(n)}^*$, while the alternating signs of elements in $\ext(B_{V_1^*})$ will restrict the signs that are allowed in the description of $\isom(V_1)$.

\begin{prop}\label{Prop: S(en) has finite supp in V_1*}
    Suppose that $S:V_1^*\rightarrow V_1^*$ is a weak*-weak* continuous operator that preserves extreme points. Then, for all $n\in \N$, $S(e_n^*)$ is a compatible functional.
\end{prop}
\begin{proof}
    For every $n\in \N$, using Theorem~\ref{teo: extreme points in BV1*}, it is straightforward to check that $e_n^*\in w^*$-$\cl\big(\ext B_{V_1^*}\big)$. Furthermore, because the set of compatible functionals is weak* closed and contains the extreme points of $B_{V_1^*}$, we have:
    \begin{align*}
        S(e_n^*) \in S\Big[ w^*\text{-}\cl\left(
    \ext\big( B_{V_1^*}\big)\right)\Big] &\subseteq 
     w^*\text{-}\cl\left( S\left[
    \ext\big( B_{V_1^*}\big)\right]\right)\\ &\subseteq
    w^*\text{-}\cl\left(\ext\big( B_{V_1^*}\big)\right) \subseteq
    w^*\text{-}\cl\left(M\right) = M. \qedhere
    \end{align*}
\end{proof}

\begin{prop}\label{prop: S(extmax)subset extmax}
    Let $S\in\isom(V_1^*)$ be a $w^*$-homeomorphism. For every $x^* \in \ext (B_{V_1^*})$:
    \[
    \supp x^*\in\permax \iff \supp S(x^*)\in\permax .
    \]
\end{prop}
\begin{proof}
    Let $x^*=\sum\limits_{i=1}^{k+1}\lambda_i e_{n_i}^* \in \ext(B_{V_1^*})$ such that $\supp x^* \in \permax$. Then, $y^*=S(x^*) = \sum\limits_{i=1}^{k'} \delta_i e_{m_i}^* \in \ext (B_{V_1^*})$. 
    Suppose, towards a contradiction, that $\supp y^* \in \per \setminus \permax$. Let 
    $$\delta_{k'+1} = \left\{\begin{array}{ll}
         - \sgn \delta_{k'} & \text{ if }y^*= \pm e_1 \text{ or } |\supp y^*|=m_1\\
         - 2 \sgn \delta_{k'} & \text{ if }y^*= \pm e_2 \text{ or } |\supp y^*|=m_1-1 
    \end{array} \right.$$
    Notice that, for every $l > \supp y^*$, $y^* + \delta_{k'+1}e_l^* \in \ext (B_{V_1^*})$, so that $x^* + \delta_{k'+1}S^{-1}(e_l^*)= S^{-1}(y^* + \delta_{k'+1}e_l^*) \in \ext (B_{V_1^*})$. Moreover, since $(e^*_l)_{l}$ converges weak$^*$ to $0$, so does $(S^{-1}(e^*_l))_{l}$. Hence, there is an increasing sequence $(l_i)_i$ such that $(\supp S^{-1}(e^*_{l_i}))_{i}$ is pairwise disjoint and $\supp S^{-1}(e^*_{l_i}) > \supp x^*$ for every $i$. This implies that $\supp S^{-1}(y^* + \delta_{k'+1}e_{l_i}^*) = \supp x^* \cup \supp S^{-1}(e^*_{l_i}) \notin \per$, contradicting the fact that  $S^{-1}(y^* + \delta_{k'+1}e_l^*) \in \ext (B_{V_1^*})$ for $l_i > \supp y^*$. Thus, $\supp S(x^*)\in\permax$. Observe that the reverse implication follows from interchanging the roles of $S$ and $S^{-1}$.
\end{proof}

\begin{lemma}\label{lemma:soma}
    Let $y^*=\sum\limits_{i=1}^k\lambda_i e_{n_i}^*$ be a compatible functional such that $|\supp y^*|=n_1>1$ and $|\lambda_k|=2$. If $z^*$ is a compatible functional such that $y^* + z^* \in \ext (B_{V_1^*})$ and $\supp (y^* +z^*) \in \permax$, then $z^* = \pm e_l^*$ for some $l > n_k$.    
\end{lemma}
\begin{proof}
First, notice that $y^*$ and $z^*$ being compatible, if $|(y^*+z^*)(e_i)|=1$, then either $|y^*(e_i)|=1$ and $|z^*(e_i)|\in\{0,2\}$ or $|z^*(e_i)|=1$ and $|y^*(e_i)|\in\{0,2\}$. Remark also that if $y^* + z^* \in \ext (B_{V_1^*})$ and $\supp (y^* +z^*) \in \permax$, then it has exactly two coordinates with absolute value equal to $1$, $\min \supp (y^*+z^*)$ and $\max \supp (y^*+z^*)$.

Let $z^*= \sum\limits_{i=1}^{k'}\delta_i e_{m_i}^*$. Let us first notice that $|\delta_{k'}|=1$. Indeed, if $|\delta_{k'}|=2$, then $\max \supp (y^*+z^*) \in\{m_1, n_1\}$. If $\max \supp (y^*+z^*) = m_1$, then $n_1 < m_1 < m_{k'} = n_k$ and the coordinates of $y^*$ and $z^*$ above $m_1$ cancel each other, so that $|\supp (y^*+z^*)| \leq |\supp y^*|$. Since $\supp y^* \notin \permax$ and $\min \supp (y^*+z^*)=\min \supp y^*=n_1$, we get that $\supp (y^*+z^*) \notin \permax$. If $\max \supp (y^*+z^*) = n_1$, similarly we get that $\supp (y^*+z^*) \notin \permax$, both cases contradicting the maximality of $\supp (y^*+z^*)$. Therefore, $|\delta_{k'}|=1$.

Also, if $k'\geq 2$, we would get either three coordinates with absolute value equal to $1$, $\{n_1,n_k,m_1\}$, or two of them would be equal and the only remaining one would have absolute value equal to $1$, both cases contradicting the fact that $y^*+z^*$ has exactly two coordinates with absolute value equal to $1$. Therefore, $k'= 1$ and it is easy to see that $m_1>n_k$ in this case, which concludes the proof.
\end{proof}

In the next result we will prove that $\isom(V_1)$ is trivial and this is obtained through a characterization of some isometries in $\isom(V_1^*)$, as one would expect. We would like to highlight how important Lemma~\ref{lemma:soma} is to control the interaction between the supports of two compatible functionals of the form $S(e_1^*)$ and $S(e_l^*)$. 

\begin{theo}\label{theo: V1 has trivial group of isometries}
If $S\in\isom(V_1^*)$ is a weak$^*$-homeomorphism, then $S=\pm Id$. In particular, we have $\isom(V_1) = \{Id,-Id\}$. 
\end{theo}
\begin{proof}
Let $S(e_1^*) = \sum\limits_{i=1}^k\lambda_i e_{n_i}^*$ and we prove the following claim:

\textbf{Claim 1.} $|\supp\big(S(e_1^*)\big)| = n_1$.
\begin{proof}[Proof of Claim 1.]
For every $l>1$, we have that $e_1^* - e_l^* \in \ext (B_{V_1^*})$ and $\{1,l\} \in \permax$,
so that $S(e_1^* - e_l^*) \in \ext (B_{V_1^*})$ and $\supp S(e_1^*-e_l^*) \in \permax$, by Proposition~\ref{prop: S(extmax)subset extmax}. Since $(e^*_l)_{l}$ converges weak$^*$ to $0$, so does $(S(e^*_l))_{l}$. Hence, there is an increasing sequence $(l_i)_i$ such that $(\supp S(e^*_{l_i}))_{i}$ is pairwise disjoint. Fix $i$ such that $\supp S(e^*_{l_i}) > \supp S(e_1^*)$. Since $S(e_{l_i}^*)$ is a compatible functional, their first coordinate has absolute value equal to $1$. Also, $\supp S(e_1^* - e_l^*) \in \permax$, so that $S(e_1^* - e_l^*)$ has only two coordinates with absolute value equal to $1$ - the minimum and the maximum of its support. Hence, we conclude that $|\supp S(e_{l_i}^*)|=1$, so that $|\supp S(e_1^*)|=n_1$.
\end{proof}

Let us now show that $S$ is induced by a signed permutation. It suffices to prove the following claim:

\textbf{Claim 2.} $|\supp S(e_l^*)|=1$ for every $l \in \mathbb{N}$. 
\begin{proof}[Proof of Claim 2.]
Recall that $y^* = S(e_1^*) = \sum\limits_{i=1}^k\lambda_i e_{n_i}^*$ is an extreme point and it follows from Claim 1 that $|\supp y^*|=n_1$. Suppose, towards a contradiction, that $n_1>1$, in particular $|\lambda_k|=2$ because $\supp y^*\notin\permax$.

Given $l\geq 2$ and let $z_l^*= S(e_l^*)$. It follows from Proposition~\ref{Prop: S(en) has finite supp in V_1*} that $z_l^*$ is a compatible functional and observe that $y^*-z_l^*=S(e_1^*-e_l^*)\in\ext (B_{V_1^*})$ and $\supp(y^*-z_l^*) \in \permax$ by Proposition~\ref{prop: S(extmax)subset extmax}. Now, Lemma \ref{lemma:soma} implies that $z_l^*=\pm e_{m_l}^*$ for some $m_l > n_k$. 

The reader may now notice that $\ran(S) \subseteq \overline{span} (\{e^*_{m_l}:l \geq 2\} \cup \{y^*\})$, which yields that $e_1^*\notin \ran(S)$, a contradiction. Therefore $|\supp y^*|=n_1=1$ and we can write $y^*=\delta_1 e_{n_1}^*$. For all $l\geq 2$, in order to $\supp S(e_1^*-e_l^*)\in\permax$, we must have $|\supp S(e_l^*)|=1$ and this concludes Claim 2.
\end{proof}

Claim 2 implies that there is a bijection $\sigma$ of $\mathbb{N}$ such that for each $n\in\N$, $S(e_n^*)=\epsilon_n e_{\sigma(n)}^*$, where $|\epsilon_n|=1$. We fix such $\sigma$ for the rest of the proof.

\textbf{Claim 3.} For all $n\in\N$, $\epsilon_n = \epsilon_1$ and $\sigma(n) = n$.
\begin{proof}[Proof of Claim 3.]
Observe that the signs of the coordinates of $S(e_1^*-e_n^*) = \epsilon_1 e_1^* - \epsilon_n e_{\sigma(n)}^*$ are alternating, thus $\epsilon_n = \epsilon_1$. Moreover, for every $n \in \mathbb{N}$, let $n=m_1<\cdots<m_{n+1}$ be such that $\sigma(m_1)<\cdots<\sigma(m_{n+1})$ and notice that there is $x^*\in \ext (B_{V_1^*})$ such that $\supp x^*= \{m_1, \dots, m_{n+1}\} \in \permax$. Hence, 
$$S(x^*)= \sum\limits_{i=1}^{n+1}\lambda_i S(e_{m_i}^*)= \sum\limits_{i=1}^{n+1}\lambda_i e_{\sigma(m_i)}^*  \in \ext (B_{V_1^*}),$$
and $\{\sigma(m_1), \dots, \sigma(m_{n+1})\} = \supp S(x^*) \in \permax$. We conclude that $\sigma(n)=\sigma(m_1)= \min \supp S(x^*) = |\supp S(x^*)|-1=n$.
\end{proof}

From what we have done, we conclude that $S=\pm Id$. Now, let $T\in\isom(V_1)$, it follows that $T^* = \pm Id_{V_1^*}$ and we have:
\[
e_m^*(T(e_n)) = (e_m^*\circ T)(e_n) = T^*(e_m^*)(e_n) = \pm e_m^*(e_n),
\]
and we conclude that $T=\pm Id_{V_1}$.
\end{proof}
\section{A rigidity result on Lorentz sequence spaces}
\label{sec lorentz space}
Given a decreasing sequence $w=(w_n)_{n\geq 1}$ of positive real numbers  such that $w\in c_0\setminus \ell_1$, we define the Lorentz sequence space, denoted by $d(w,1)$, as follows:
\[
d(w,1)=\{ x=(x_n)_{n\geq 1} \in \R^\N: \|x\|_{w,1}<\infty \},
\]
 where $\|x\|_{w,1} = \sup\limits_{\pi \in S_\infty} \sum\limits_{i=1}^\infty |x_{\pi(n)}|w_n$ and $S_\infty$ denotes the set of all permutations of natural numbers.
 In this case, $w$ is called a sequence of weights, the space $d(w,1)$ endowed with $\|\cdot\|_{w,1}$ is a Banach sequence space and the canonical vectors $(e_n)_{n\geq1}$ form a symmetric Schauder basis. We will omit the index in the norm to improve readability.

 It is well known that $\|x\|$ is attainend by its decreasing rearrangement, i.e., the sequence $\widetilde x = (\widetilde x_n)_{n\geq 1}$ that is obtained from rearranging $(|x_n|)_{n\geq 1}$ as a decreasing sequence via a permutation.  In what follows, we will use the following notation: $W(n) = \sum\limits_{i=1}^n w_i$. The reader may notice that because $w\notin \ell_1$ is decreasing, we cannot have $w_n = 0$, thus $(W(n))_{n\geq 1}$ is strictly increasing. This will be important in the following results. We also include the following characterization of the extreme points of $d(w,1)$:

 \begin{theo}[{\cite[Theorem 2.6]{AnnaLeeExtreme}}] \label{Theo: extremal Anna}
 An element  $x\in S_{d(w,1)}$ is an extreme point of $B_{d(w,1)}$ if, and only if, there exists $n_0\in\N$ such that:
\[
\widetilde x_n = \begin{cases}
\dfrac{1}{W(n_0)}&,  \text{ if } n\leq n_0\\
0 &, \text{ if } n>n_0
\end{cases}
\]
and $w_1>w_{n_0}$, when $n_0>1$.
 \end{theo}

 Without passing through the decreasing rearrangement $\widetilde x$ of $x$, the previous result states that $x\in \ext B_{d(w,1)}$ if, and only if, $x = \tfrac{1}{W(n_0)}\epsilon \rchi_{A}$, where $|A| = n_0$, $\epsilon \in \{-1,1\}^A$ and, when $A$ is not a singleton, $w_{n_0}<w_1$. Although this notation is longer than the one used in Theorem~\ref{Theo: extremal Anna}, the computations we will carry out in this section are done using the support of the elements in $d(w,1)$, thus the need of introducing the signs of each non-null coordinate.

Extending the language of \cite{AntunesBeanlandsurvey}, given two Banach sequence spaces $X$ and $Y$ and $\lambda\in\R$, we say that $\isom(X,Y)$ is $\lambda$-\emph{standard} if it is composed by all $T:X\rightarrow Y$ given by 
\begin{equation}\label{eq: standard group of isometries def}
   T(x_n) = (\lambda\epsilon_n x_{\pi(n)})
\end{equation}
where $\epsilon = (\epsilon_n)_n \in \{-1,1\}^\N $ is a sequence of signs and $\pi\in S_\infty$ is a permutation. We say that an operator satisfying equation~\eqref{eq: standard group of isometries def} is induced by the scalar $\lambda$, the permutation $\pi$ and the sequence of signs $\epsilon$. Notice that if $X=Y$ and $\lambda=1$, this coincides with the definition of standard group of isometries.

In this section we aim to prove that $\isom\big(d(v,1),d(w,1)\big)$ is either empty or $\lambda$-standard, depending on the weight sequences. As a consequence, we obtain a Banach-Stone like result for Lorentz sequences spaces that is similar to the one presented in \cite{CarothersIsometrisLorentz} for the Lorentz function spaces $L_{w,p}$ and, the space $d(w,1)$ and it's predual $d_*(w,1)$ have standard group of isometries. In our arguments, we will make use of certain vectors which are supported in 2 or 3 elements of the canonical basis and we will need them to be extreme points of the unit ball. This is guaranteed for spaces whose sequence of weights $w$ is strictly decreasing. Hence, from now on we shall work under this assumption. In particular, we get from the previous theorem that
\begin{equation*}ext(B_{d_{w,1}}) = \left\{ \frac{1}{W(n)} \epsilon \chi_A: n \in \mathbb{N}, |A|=n \text{ and } \epsilon \in \{-1,1\}^A \right\}.
\end{equation*}

\begin{lemma}\label{lemmav2: supp igual ou disjunto}
    If $T\in \isom\big(d(v,1);d(w,1)\big)$, then, for all $r,s\in\N$,
    \[
    \supp T(e_r) = \supp T(e_s)
    \qquad\text{or}\qquad
    \supp T(e_r)\cap \supp T(e_s) = \emptyset.
    \]
\end{lemma}
\begin{proof}
For each $r \in \mathbb{N}$, by Theorem~\ref{Theo: extremal Anna}, there is a finite $A_r\subseteq \mathbb{N}$ and $\epsilon^r\in\{-1,1\}^{A_r}$
such that $T(e_r) = \tfrac{v_1}{W(n_r)}\epsilon^r\rchi_{A_r}$. Let $n_r=|A_r|$.

Given distinct $r,s\in\N$, suppose that $A_r \cap A_s \neq \emptyset$ and fix $q\in A_r\cap A_s$. Notice that the point
\[
z=\frac{\epsilon^r(q) e_r +\epsilon^s(q)e_s}{V(2)} \in \ext B_{d(v,1)}
\]
so that
\[
T(z) = \frac{1}{V(2)}\left(\epsilon^r(q)\frac{v_1}{W(n_r)}\epsilon^r\rchi_{A_r} + \epsilon^s(q)\frac{v_1}{W(n_s)}\epsilon^s\rchi_{A_s}\right) \in \ext B_{d(w,1)}.
\] 
In the view of Theorem~\ref{Theo: extremal Anna}, all the non-zero coordinates of $T(z)$ equal to $\frac{1}{W(n)}$ in absolute value, where $n=|\supp T(z)|$. 

Now, if $A_r \neq A_s$ and, without loss of generality, there is $p\in A_r\setminus A_s$, then 
\[
|T(z)(p)| = \frac{v_1}{V(2)W(n_r)} <
\frac{v_1}{V(2)}
\left(\frac{1}{W(n_r)} + \frac{1}{W(n_s)} \right) = |T(z)(q)|,
\]
a contradiction. Hence, if $A_r \cap A_s \neq \emptyset$, then $A_r = A_s$.
\end{proof}

\begin{lemma} \label{lemma: supp T igual supp T-1}
    If $T\in \isom\big(d(v,1);d(w,1)\big)$, then $|\supp T(e_r)| = |\supp T(e_s)|$ for all $r,s\in\N$.
\end{lemma}
\begin{proof}
For each $r \in \mathbb{N}$, by Theorem~\ref{Theo: extremal Anna}, there is a finite $A_r\subseteq \mathbb{N}$ and $\epsilon^r\in\{-1,1\}^{A_r}$
such that $T(e_r) = \tfrac{v_1}{W(n_r)}\epsilon^r\rchi_{A_r}$. Let $n_r=|A_r|$.

Suppose, towards a contradiction, that 
 for some distinct $r$ and $s$, $n_r<n_s$. Then, by Lemma~\ref{lemmav2: supp igual ou disjunto}, $A_r\cap A_s = \emptyset$.

Fix $a>0$ and let $0<b< \dfrac{a W(n_r)}{W(n_s)}<a$. Notice that $\dfrac{b}{W(n_s)}<\dfrac{a}{W(n_r)}$ and $\dfrac{b}{W(n_r)}< \dfrac{a}{W(n_s)}$, thus, we can compute:

\begin{equation*}
    av_1+b v_2 = \| ae_r + b e_s\|_{v,1} = \| aT(e_r) + b T(e_s)\|_{w,1} = av_1 + \frac{bv_1}{W(n_s)}( W(n_s+n_r) - W(n_r))
\end{equation*}
and
\begin{equation*}
    av_1+bv_2 = \| ae_s + b e_r\|_{v,1} = \| a T(e_s)+ b T(e_r)\|_{w,1} = av_1 + \dfrac{bv_1}{W(n_r)} (W(n_r+n_s) - W(n_s))
\end{equation*}

Manipulating the former equations yields:
\begin{equation*} 
  \Big( W(n_r+n_s) - W(n_s) - W(n_r) \Big)\cdot\Big( W(n_s)-W(n_r)\Big)=0.  
\end{equation*}

However, $W(n_r)<W(n_s)$ because $n_r< n_s$ implies the following:
\[
W(n_r + n_s) = w_1+\cdots + w_{n_s} + \underbrace{ w_{n_s+1}+\cdots w_{n_s+n_r}}_{<W(n_r)} < W(n_s) + W(n_r),
\]
a contradiction.
\end{proof}

\begin{lemma}\label{lemmav2: supp disjunto}
    If $T\in \isom\big(d(v,1);d(w,1)\big)$, then:
    \[
    \supp T(e_r)\cap \supp T(e_s) = \emptyset
    \]
and, therefore, $|\supp T(e_r)|=1$ for every $r \in \mathbb{N}$.
\end{lemma}
\begin{proof}
Suppose $r,s \in \mathbb{N}$ are distinct such that $\supp T(e_r)= \supp T(e_s)$ and let $t$ be such that $\supp T(e_r) \cap \supp T(e_t) = \emptyset$. Let $n= |\supp T(e_r)| = |\supp T(e_t)|$. Hence, the nonzero coordinates of $T(e_r),T(e_s),T(e_t)$ have absolute value equal to $\frac{v_1}{W(n)}$.

Then
$$\frac{1}{V(3)} (e_r + e_s + e_t) \in \ext B_{d(v,1)}.$$
It follows that
$$\frac{1}{V(3)} \left(T(e_r) + T(e_s) + T(e_t)\right) \in \ext B_{d(w,1)},$$
which implies that all nonzero coordinates of $T(e_r) + T(e_s) + T(e_t)$ have the same absolute value. 

On the other hand, it is easy to see that the nonzero coordinates of $T(e_r) + T(e_s)$ have absolute value equal to $\frac{2v_1}{W(n)}$, while the nonzero coordinates of $T(e_t)$ have absolute value equal to $\frac{v_1}{W(n)}$. Since $\supp (T(e_r) + T(e_s))$ is disjoint from $\supp T(e_t)$, we get coordinates of $T(e_r) + T(e_s) + T(e_t)$ with different absolute values, a contradiction. 
\end{proof}

We finally use the previous lemmas to prove the main result of this section:

\begin{theo}\label{theo: iso dv,dw are induced by permutation and scalar}
     Let $v,w\in c_0\setminus \ell_1$ be stritcly decreasing. If $\isom\big(d(v,1);d(w,1)\big)$ is non-empty, then there is $\lambda >0$ such that $v = \lambda w$ and $\isom\big(d(v,1);d(w,1)\big)$ is $\lambda$-standard.
    \end{theo}
\begin{proof}
    Let $T\in\isom\big(d(v,1);d(w,1)\big)$. By Lemma~\ref{lemmav2: supp disjunto}, there is $\pi \in S_\infty$ such that, for every $n\in\N$, $T(e_n) = \lambda_{n} e_{\pi(n)}$ and since $T$ is an isometry:
    \[
    v_1 = \|e_n\|_{v,1} = \|T(e_n)\|_{w,1} = \|\lambda_n e_{\pi(n)}\|_{w,1} = |\lambda_{n}|w_1,
    \]
hence, for all $n\in\N$, $|\lambda_{n}| = \dfrac{v_1}{w_1}$ and we conclude that 
$T(e_n) = \lambda\epsilon_{n} e_{\pi(n)}$ with $(\epsilon_{n})\in\{-1,1\}^\mathbb{N}$ and $\lambda=\frac{v_1}{w_1}$.

Now, by induction on $n$ we prove that $\frac{v_n}{w_n} = \lambda$. Assume this holds for $n \in \mathbb{N}$, so that $V(n) = \lambda W(n)$. Notice that $\Vert \sum_{i=1}^{n+1} e_i \Vert_{v,1} = V(n+1)$, while
$$\Vert T(\sum_{i=1}^{n+1} e_i)\Vert_{w,1}  = \lambda \Vert \sum_{i=1}^{n+1} \epsilon_i e_{\pi(i)}\Vert_{w,1}  =
\lambda W(n+1).$$
Hence, $V(n) + v_{n+1} = V(n+1) = \lambda W(n+1) = \lambda (W(n) + w_{n+1})$, which implies that $v_{n+1} = \lambda w_{n+1}$. 

Finally, it is easy to see that in case $v=\lambda w$,  every operator $T:d(v,1) \rightarrow d(w,1)$ defined by $ T(e_n) = \lambda\epsilon_n e_{\pi(n)}$ is an isometry.
\end{proof}

It is easy to check that the standard unitary vectors $(e_n)$ form a boundedly complete Schauder basis for $d(w,1)$, thus it admits a predual. It is well known that the predual of the Lorentz sequence spaces is:
\[
d_*(w,1) = \left\{x =(x_n) \in \R^\N: \lim_n \frac{1}{W(n)}\sum_{i=1}^{n} \widetilde x_i =0 \right\},
\]
endowed with the norm $\|(x_n)\|_W = \sup\limits_n \dfrac{\sum_{i=1}^{n} \widetilde x_i}{W(n)}$. The Lorentz sequence space is a particular case of the Köthe space $\mu_w$, and in the section 9 of \cite{garling1966symmetric}, the author gives a description of both the dual and the predual of $d(w,1)$ passing by the $\alpha$-dual space $\mu_w^\times$ (defined by Köthe and Toeplitz in \cite{koethe1934lineare}), although no explicit isometry is given and we couldn't find any references in the literature that describes the above identification via an isometry. On the other hand, on recent papers the dual and the predual of different constructions of Lorentz spaces are presented in \cite{kaminska2004mideal} and in \cite{ciesielskisequencelorentzstructure} and we decided to include a proof here for the sake of completeness:

\begin{prop}\label{prop: isometria predual lorentz}
    The operator  $\Phi_w: d_*(w,1)^* \rightarrow d(w,1)$ given by $\Phi_w(f) = ( f(e_n) )_{n\geq 1}$ is an isometry.
\end{prop}
\begin{proof}
    We will check that $\Phi_w$ is well defined.  Fix $f\in d_*(w,1)^*$, denote $a_i = f(e_i)$ and for each $\pi\in S_\infty$ and $n\in\N$, consider $y(\pi,n) = \sum\limits_{i=1}^n \sgn(a_{\pi(i)})w_i e_{\pi(i)}$. Then $y(\pi,n)\in S_{d_*(w,1)}$ and:
\[
\forall n\in\N:\quad \sum_{i=1}^n |a_{\pi(i)}|w_i = f(y(\pi,n)) \leq \|f\|
\]
Thus, $\Phi_w$ well defined and $\|\Phi_w(f)\|\leq \|f\|$. On the other hand, if $y=(y_n)\in d(w,1)$, we will check that $f:d_*(w,1)\rightarrow \R$ by $f(x) = \sum\limits_{n=1}^\infty x_n y_n$, for $x=(x_n) \in d_*(w,1)$, is a bounded functional. 

If $x = (x_n) \in d_*(w,1)$, by the definition of the norm,  $S(n)\doteq \sum\limits_{i=1}^{n} \widetilde x_i\leq W(n)\|x\|_W $ for each $n\in \N$. Thus:
\begin{align*}
    \sum_{i=1}^n \widetilde y_i \, \widetilde x_i &=
        \sum_{i=1}^{n-1} \Big[\Big( \widetilde y_i - \widetilde y_{i+1}\Big)S(i)\Big] + \widetilde y_n S(n) \\
        &\leq \|x\|_W \sum_{i=1}^{n-1} \Big[\Big( \widetilde y_i - \widetilde y_{i+1}\Big)W(i)\Big] + \widetilde y_nW(n)\\
        &= \|x\|_W \sum_{i=1}^n \widetilde y_i w_i \leq \|x\|_W \|y\|_{w,1}.
\end{align*}

We conclude that $f$ is well defined and $\|f\| \leq \|y\|_{w,1}$. Therefore, $\Phi_w$ is a surjective isometry.
\end{proof}

The next result shows that the isometries between preduals of Lorentz spaces inherit the rigidity aspects of their dual spaces:

\begin{theo}\label{theo: isometria standard predual lorentz}
    Let $v,w\in c_0\setminus \ell_1$ be stritcly decreasing. If $\isom\big(d_*(v,1);d_*(w,1)\big)$ is non-empty, then there is $\lambda >0$ such that $v = \lambda w$ and $\isom\big(d_*(v,1);d_*(w,1)\big)$ is $\lambda$-standard. In particular, $d_*(w,1)$ has standard group of isometries.
\end{theo}
\begin{proof}
    Let $T\in \isom\big(d_*(v,1);d_*(w,1)\big)$. Because the adjoint operator $T^*:d_*(w,1)^* \rightarrow d_*(w,1)^*$ is also an isometry, we define an isometry $S$ on $d(w,1)$ via the (commutative) diagram:
\begin{center}
    \begin{tikzcd}[column sep=1.2cm, row sep=1.1cm]
{d(v,1)}                       & {d(w,1)} \arrow[l, "S"', dashed]                 \\
{d_*(v,1)^*} \arrow[u, "\Phi_v"] & {d_*(w,1)^*} \arrow[l, "T^*"] \arrow[u, "\Phi_w"']
\end{tikzcd}
\end{center}
where $\Phi_v$ and $\Phi_w$ are given by Proposition~\ref{prop: isometria predual lorentz}. 

Notice that it follows from the definiton of $\Phi$ that $\Phi^{-1}(e_n) = e_n^*$ for all $n\in\N$ in each column of the diagram. By Theorem~\ref{theo: iso dv,dw are induced by permutation and scalar}, $S(z_n) = (\lambda\epsilon_n z_{\pi(n)})_{n\geq 1}$, where $\pi\in S_\infty$, $\epsilon\in \{-1,1\}^\N$ and $\lambda = \dfrac{w_1}{v_1}$. Denoting  $y=T(x)$, because the above diagram is commutative,  for all $k\in\N$:
\[
y_k = \Phi_w^{-1}(e_k)(y) = \Phi_w^{-1}(e_k) (T(x)) 
= \Big[T^*\circ \Phi_w^{-1}(e_k)\Big](x)
\]
\[=\Big[\Phi_v^{-1}\circ S(e_k)\Big](x)
= \Phi_v^{-1}(\lambda \epsilon_k e_{\pi(k)})(x)
= \epsilon_k x_{\pi(k)}.
\]

Thus, $T$ is induced by $\lambda$ and $\pi\in S_\infty$. In particular, when $v_1 = w_1$, we conclude that $v=w$ and that $d_*(w,1)$ has standard group of isometries.
\end{proof}
\bibliographystyle{plain}
\bibliography{references}
\end{document}